\documentclass[11pt]{article}
\usepackage{amsmath,amssymb,amsthm}
\usepackage[hidelinks]{hyperref}
\usepackage{cleveref}

\crefname{equation}{}{}
\Crefname{equation}{Equation}{Equations}
\crefname{theorem}{Theorem}{Theorems}
\Crefname{theorem}{Theorem}{Theorems}
\crefname{lemma}{Lemma}{Lemmas}
\Crefname{lemma}{Lemma}{Lemmas}
\crefname{proposition}{Proposition}{Propositions}
\Crefname{proposition}{Proposition}{Propositions}
\crefname{corollary}{Corollary}{Corollaries}
\Crefname{corollary}{Corollary}{Corollaries}
\crefname{conjecture}{Conjecture}{Conjectures}
\Crefname{conjecture}{Conjecture}{Conjectures}
\crefname{section}{Section}{Sections}
\Crefname{section}{Section}{Sections}
\crefname{example}{Example}{Examples}
\Crefname{example}{Example}{Examples}
\crefname{problem}{Problem}{Problems}
\Crefname{problem}{Problem}{Problems}
\crefname{table}{Table}{Tables}
\Crefname{table}{Table}{Tables}
\crefname{remark}{Remark}{Remarks}
\Crefname{remark}{Remark}{Remarks}
\crefname{definition}{Definition}{Definitions}
\Crefname{definition}{Definition}{Definitions}

\newcommand{\ZZ}{\mathbb{Z}}

\newtheorem{theorem}{Theorem}
\newtheorem{lemma}[theorem]{Lemma}

\theoremstyle{definition}

\newcommand{\doi}[1]{\href{http://dx.doi.org/#1}{\texttt{doi:#1}}}
\newcommand{\arxiv}[1]{\href{http://arxiv.org/abs/#1}{\texttt{arXiv:#1}}}

\title{On the asymptotic of lottery numbers}

\author{Alexander Sidorenko\\
\small Department of Extremal Combinatorics \\
\small Alfr\'ed R\'enyi Institute of Mathematics, Hungary \\
\small\tt sidorenko.ny@gmail.com
}

\date{\today}

\begin{document}

\maketitle

\begin{abstract}
Let $L(n,k,r,p)$ denote the minimum number of $k$-subsets of an $n$-set such that 
all the $\binom{n}{p}$ $p$-subsets are intersected by one of them 
in at least $r$ elements. 
The case $p=r$ corresponds to the covering numbers, 
while the case $k=r$ corresponds to the Tur\'an numbers. 
In both cases, there exists a limit of 
$L(n,k,r,p) / \binom{n}{r}$ as $n\to\infty$. 
We prove the existence of this limit in the general case. 


\end{abstract}

A $k$-\emph{graph} with a vertex set $V$ is a system of $k$-element subsets of $V$ 
called \emph{edges}. 
Such a system is called
an $(n,k,r,p)$-\emph{lottery system} 
if 
$|V|=n$ and 
for every $p$-element subset $P \subseteq V$ 
there is an edge $K$ such that $|K \cap P| \geq r$. 
The minimum size of an $(n,k,r,p)$-lottery system, denoted by $L(n,k,r,p)$, 
was studied in \cite{
Bate:1998,
Bertolo:2004,
Cushing:2023,
Droesbeke:1982,
Furedi:1996,
Hanani:1964,
Li:1999,Li:2000,Li:2002,Montecalvo:2015,Sidorenko:2023}. 
The smallest presently known lottery systems with parameters 
$n \leq 90$, $r \leq k \leq 15$, $p \leq 20$ 
are listed on the website \cite{Italian:tables}. 

An $(n,k,r,r)$-lottery system is also known as an $(n,k,r)$-\emph{covering} system. 
Its minimum size is $C(n,k,r)=L(n,k,r,r)$. 
R\"odl \cite{Rodl:1985} proved that 
\begin{equation}\label{eq:Rodl}
  \lim_{n\to\infty} \frac{C(n,k,r)}{\binom{n}{r}} = \frac{1}{\binom{k}{r}} \, .
\end{equation}

We say that an $r$-graph has \emph{Tur\'an} $p$-\emph{property} 
if every $p$-element subset of vertices contains an edge. 
The \emph{Tur\'an number} $T(n,p,r)$ is the minimum 
number of edges in an $n$-vertex $r$-graph with this property. 
Obviously, $T(n,p,r) = L(n,r,r,p)$. 
A survey on the Tur\'an numbers can be found in \cite{Sidorenko:1995}. 
It is known that there exists a limit
\begin{equation*}
  t(p,r) = \lim_{n\to\infty} \frac{T(n,p,r)}{\binom{n}{r}} \, .
\end{equation*}

It is natural to ask whether a limit of 
$L(n,k,r,p) / \binom{n}{r}$ 
exists for any fixed $k,r,p$ as $n\to\infty$. 
In the present note, we provide an affirmative answer to this question. 

The $r$-\emph{shadow} of a $k$-graph $H$ is an $r$-graph with the same vertex-set 
whose edges are the $r$-element subsets 
contained in the edges of $H$. 
Notice that an $n$-vertex $k$-graph is an $(n,k,r,p)$-lottery system 
if and only if its $r$-shadow has the Tur\'an $p$-property. 
Since a $k$-element subset contains $\binom{k}{r}$ $r$-element subsets, 
we have a simple bound
\begin{equation}\label{eq:Turan}
  L(n,k,r,p) \:\geq\: T(n,p,r) / {\textstyle \binom{k}{r}} \, .
\end{equation}

Let ${\cal F}$ be a family of forbidden $r$-graphs. 
We call an $r$-graph ${\cal F}$-\emph{free} if it does not
contain subgraphs isomorphic to members of ${\cal F}$. 
Let $K_p^r$ denote the complete $r$-graph on $p$ vertices. 
It is easy to see that an $r$-graph has the Tur\'an $p$-property 
if and only if its complement is $\{K_p^r\}$-free. 
Denote by $L(n,k,r,{\cal F})$ the minimum number of edges in an $n$-vertex $k$-graph 
such that the complement of its $r$-shadow is ${\cal F}$-free. 
Then $L(n,k,r,p) = L(n,k,r,\{K_p^r\})$. 

We call an $r$-graph \emph{pair-covering} 
if any pair of its vertices is contained in an edge. 
In other words, an $r$-graph is pair-covering if its $2$-shadow is a complete graph. 
For instance, $K_p^r$ is pair-covering for any $p \geq r$. 

\medskip

Our main result is 

\begin{theorem}\label{th:limit}
If every member of ${\cal F}$ is a pair-covering $r$-graph, then there exists 
$\lim_{n\to\infty} L(n,k,r,{\cal F}) \, n^{-r}$.
\end{theorem}

It follows from \cref{th:limit} that for any $k,r,p$, there exists the limit 
\begin{align*}
  l(k,r,p) \: = \: \lim_{n\to\infty} L(n,k,r,p) / {\textstyle \binom{n}{r}} \, .
\end{align*}
The value of the limit is known only for $r=2$ (see \cite{Furedi:1996}): 
$l(k,2,p) = \frac{1}{p-1} \binom{k}{2}^{-1}$. 
It is highly likely that $l(4,3,p) = \frac{1}{(p-1)^2}$. 
The upper bound was proved in \cite{Sidorenko:2023} for $p=4$ 
and can easily be generalized for any $p$. 
The lower bound would follow from 
\cref{eq:Turan} and the famous 
Tur\'an's conjecture that asserts $t(p,3)=\frac{4}{(p-1)^2}$. 

\medskip

In order to prove \cref{th:limit}, we need a few auxiliary results. 

\medskip

An $N$-\emph{blow-up} of an $r$-graph $G$ is obtained 
by replacing each vertex of $G$ with $N$ clones, 
and replacing each edge of $G$ with $N^r$ edges 
that correspond to $N^r$ possible ways of replacing each original vertex 
with one of its $N$ clones. 

\begin{lemma}\label{th:blowup}
If every member of ${\cal F}$ is a pair-covering $r$-graph, 
and $G$ is ${\cal F}$-free, 
then blow-ups of $G$ are also ${\cal F}$-free. 
\end{lemma}

\begin{proof}[\bf{Proof}]
Suppose to the contrary that 
a blow-up of $G$ contains a subgraph isomorphic to $F\in{\cal F}$. 
This subgraph can not contain two clones of the same vertex of $G$ since 
any two vertices of $F$ are contained in some edge of $F$, 
while
no edge in the blow-up contains two clones of the same vertex. 
Then $G$ must also contain a copy of $F$, 
which contradicts the assumption that $G$ is ${\cal F}$-free. 
\end{proof}

A \emph{Vandermonde matrix} is a square matrix $[a_{ij}]$ 
where $a_{ij} = (x_j)^i$, \;($i,j=0,1,\ldots,n-1$). 
It is well known that its determinant is equal to 
\begin{align*}
      D(x_0,x_1,\ldots,x_{n-1}) = \prod_{0 \leq i < j< n} (x_j - x_i) \, . 
\end{align*}
Notice that $D(x_0,x_1,\ldots,x_{n-1}) > 0$ 
when $x_0 < x_1 < \ldots < x_{n-1}$. 
Let $M_{k,n}$ denote the least common multiple of determinants 
$D(x_0,x_1,\ldots,x_{n-1})$ 
over all choices of integer values 
$0 \leq x_0 < x_1 < \ldots < x_{n-1} < k$. 

\medskip

We will consider matrices over the ring of integers modulo $N$ 
which, for simplicity, will be denoted as $\ZZ_N$. 
It is easy to see that if $N = aM_{k,n}+1$ with $a \geq 1$, 
and $x_0,x_1,\ldots,x_{n-1}$ are pairwise distinct integers from $\{0,1,\ldots,k-1\}$, 
then the Vandermonde matrix with entries $a_{ij}=(x_j)^i$ 
is invertible over $\ZZ_N$. 

\begin{lemma}\label{th:GDD}
Let $N = aM_{k,n}+1$ with $a \geq 1$. 
There exists a $k$-partite $k$-graph $G_{N,k,r}$ 
with $N$ vertices in each part
such that 
for any $r$ vertices from $r$ distinct parts 
there is exactly one edge that contains them. 
\end{lemma}

\begin{proof}[\bf{Proof}]
Consider a $(k-r) \times k$ matrix ${\bf A}=[j^i]$ 
where $i=0,1,\ldots,k-r-1$, $j=0,1,\ldots,k-1$. 
Each of its $(k-r) \times (k-r)$ submatrices is a Vandermonde matrix 
invertible over $\ZZ_N$. 
Let the vertex-set of $G_{N,k,r}$ consist of pairs $(j,z)$ where 
$j\in\{0,1,\ldots,k-1\}$ and $z\in\ZZ_N$.
Vertices $(0,z_0),(1,z_1),\ldots,(k-1,z_{k-1})$ form an edge if and only if 
\begin{align}\label{eq:lemma}
  {\bf A} \cdot (z_0,z_1,\ldots,z_{k-1})^T = {\bf 0} 
  \;\;\;{\rm over}\;\: \ZZ_N \, . 
\end{align}
If we fix values of any $r$ variables among $z_0,z_1,\ldots,z_{k-1}$, 
then there is exactly one solution of \cref{eq:lemma} 
for the remaining $k-r$ variables. 
\end{proof}

\begin{proof}[\bf{Proof of \cref{th:limit}}].
Set $\Delta := k-r$.
It follows from \cref{eq:Rodl} that there exist $c_{k,r}$ and  $N_*$ such that 
$C(n,k',k'-\Delta) \leq c_{k,r} n^{k'-\Delta}$
for every $n \geq N_*$ and every $k' \in [\Delta+2, k]$, 
Set $\gamma := \liminf_{n\to\infty} L(n,k,r,{\cal F}) \, n^{-r}$. 
We are going to prove that for every $\varepsilon > 0$, there is $n_{\varepsilon}$ such that 
$L(n,k,r,{\cal F}) \, n^{-r} \leq \gamma + \varepsilon$ 
for all $n \geq n_{\varepsilon}$. 

Select $m \geq \frac{4}{\varepsilon} k c_{k,r} \frac{(r-2)^{r-2}}{(r-2)!}$ 
such that 
$L(m,k,r,{\cal F}) \, m^{-r} \leq \gamma + \frac{\varepsilon}{4}$. 
Select $N_{\varepsilon} \geq N_*$ such that 
$(1+M_{k,\Delta}/N_{\varepsilon})^r \leq
\frac{\gamma+\varepsilon}{\gamma+\frac{\varepsilon}{2}}$. 
Set $n_{\varepsilon} := mN_{\varepsilon}$. 
Let $n \geq n_{\varepsilon}$. 
There exists $N \equiv 1 \bmod M_{k,\Delta}$ such that $n \leq mN < n+mM_{k,\Delta}$. 
Note that $N \geq N_{\varepsilon}$. 

Consider a $k$-graph $H$ with $m$ vertices and $L(m,k,r,{\cal F})$ edges 
such that the complement of its $r$-shadow is ${\cal F}$-free. 
Let $V$ be its vertex-set. 
Set $V_N := V\times\ZZ_N$. 
Let $e=\{v_1,\ldots,v_k\}$ be an edge of $H$. 
By \cref{th:GDD}, we can construct a $k$-graph $G_e$ with the vertex-set 
$\{(v_i,z):\: i=1,\ldots,k,\: z\in\ZZ_N\}$ and $N^r$ edges such that 
for any $1 \leq i_1 < i_2 < \ldots < i_r \leq k$ and any $z_1,z_2,\ldots,z_r\in\ZZ_N$, 
there is exactly one edge in $G_e$ that contains vertices 
$(v_{i_1},z_1), (v_{i_2},z_2),\ldots,(v_{i_r},z_r)$. 
Let ${\cal A}$ denote the union of the edge-sets of $k$-graphs $G_e$ 
over all edges $e$ in $H$. 
Then 
\begin{align*}
  |{\cal A}| \: = \: N^r L(m,k,r,{\cal F}) 
  \:\leq\: \left(\gamma + \frac{\varepsilon}{4}\right) (mN)^r \, .
\end{align*} 

Let $\pi$ denote the natural projection of $V_N$ on $V$, 
so $\pi((v,z)) = v$. 
For a vertex $v$ of $H$, denote $X_v =\{(v,z): z\in\ZZ_N\}$. 
For each $k' \in [\Delta+2,k]$, 
construct a $k'$-graph $C_{v,k'}$, 
which is an $(N,k',k'-\Delta)$-covering system on $X_v$ of size $C(N,k',k'-\Delta)$. 
Let ${\cal B}_{v,k'}$ denote the set of all $k$-element subsets $A \subset V_N$ 
such that $A \cap X_v$ is an edge of $C_{v,k'}$ and $|\pi(A \backslash X_v)| \leq r-2$. 
Then 
\begin{align*}
  |{\cal B}_{v,k'}| & \:\leq\: C(N,k',k'-\Delta) \binom{m-1}{r-2} \binom{(r-2)N}{k-k'}
  \\ & \:\leq\: c_{k,r}\, N^{k'-\Delta}\, \frac{m^{r-2}}{(r-2)!}\, (r-2)^{r-2}\, N^{k-k'}
  \\ & \: = \: c_{k,r}\, \frac{(r-2)^{r-2}}{(r-2)!}\, m^{r-2}\, N^r \, .
\end{align*}
Let ${\cal B}$ denote the union of ${\cal B}_{v,k'}$ over 
all $k' = \Delta+2,\ldots,k$ 
and 
all vertices $v$ of $H$. 
Then 
\begin{align*}
  |{\cal B}| & \:\leq\: m \cdot k \cdot c_{k,r}\, \frac{(r-2)^{r-2}}{(r-2)!}\, m^{r-2}\, N^r
  \\ & \:\leq\: k\, c_{k,r}\, \frac{(r-2)^{r-2}}{(r-2)!}\, m^{r-1}\, N^r
  \\ & \:\leq\: \frac{\varepsilon}{4}\, (mN)^r \, .
\end{align*}
We claim that every $r$-element subset $B \subseteq V_N$ with $|\pi(B)| \leq r-1$ 
is contained in one of the members of ${\cal B}$. 
Indeed, since $|B| > |\pi(B)|$, there is vertex $v$ such that $t=|B \cap X_v| \geq 2$. 
Then $B$ is contained in one of the members of ${\cal B}_{v,t+\Delta}$. 

Let $H_N$ be a $k$-graph with the vertex-set $V_N$ 
and the edge-set ${\cal A} \cup {\cal B}$. 
Then the complement of the $r$-shadow of $H_N$ is contained 
in the $N$-blow-up of the complement of the $r$-shadow of $H$. 
Therefore, by \cref{th:blowup}, 
the complement of the $r$-shadow of $H_N$ is ${\cal F}$-free, 
and 
$L(mN,k,r,{\cal F} \leq |{\cal A}| + |{\cal B}| \leq 
\left(\gamma + \frac{\varepsilon}{2}\right) (mN)^r$. 
As $n \leq mN \leq n+mM_{k,\Delta}$ and $N \geq N_{\varepsilon}$, we get 
\begin{align*}
  (mN/n)^r \:\leq\: ((n+mM_{k,\Delta})/n)^r 
  \:\leq\: (1+M_{k,\Delta}/N_{\varepsilon})^r 
  \:\leq\: \frac{\gamma+\varepsilon}{\gamma+\frac{\varepsilon}{2}}
\end{align*}
and 
\begin{align*}
  L(n,k,r,{\cal F}) \leq L(mN,k,r,{\cal F}) 
  \:\leq\: \left(\gamma + \frac{\varepsilon}{2}\right) (mN)^r 
  \:\leq\: (\gamma + \varepsilon) n^r \, .
\end{align*}

\end{proof}

\end{document}